\documentclass{amsart}
\usepackage{graphicx}

\usepackage{tikz}
\usetikzlibrary{automata}

\usepackage[autostyle, english = american]{csquotes}
\MakeOuterQuote{"}

\newtheorem{thm}{Theorem}[section]

\newtheorem{lem}[thm]{Lemma}
\newtheorem{prop}[thm]{Proposition}

\theoremstyle{definition}
\newtheorem{defn}[thm]{Definition}
\theoremstyle{remark}

\newtheorem{ex}[thm]{Example}
\numberwithin{equation}{section}

\begin{document}

\title[Circuit bases for randomisation]{Circuit bases for randomisation}%

\author{Elena Pesce}%
\address{Swiss Re Institute}%
\email{elena\_pesce@swissre.com}%

\author{Fabio Rapallo}%
\address{Universit\`a di Genova}%
\email{fabio.rapallo@unige.it}%

\author{Eva Riccomagno}%
\address{Universit\`a di Genova}%
\email{riccomagno@dima.unige.it}%

\author{Henry P. Wynn}%
\address{London School of Economics}%
\email{h.wynn@lse.ac.uk}%


\begin{abstract}
After a rich history in medicine, randomisation control trials both simple and complex are in increasing use in other areas such as web-based AB testing and planning and design decisions. A main objective is to be able to measure parameters, and contrasts in particular, while guarding against biases from hidden confounders. After careful definitions of classical entities such as contrasts, an algebraic method based on circuits is introduced which gives a wide choice of randomisation schemes.

\medskip


\medskip


\end{abstract}

\maketitle



\section{Introduction}

There are ways in which a regression model can be biased because of the neglect of hidden variables, sometimes called hidden confounders. To some degree these biases can be removed using randomisation. A major source of conceptual difficulties is the continuing distinction between passive observation, characterised by the terms ``observational study'' and controlled experiment. In addition  this distinction is flavored by different intellectual traditions. In most fields a controlled experimental design is conceived as an {\em intervention}. Thus one talks about setting the level of a variable $X$, or applying a treatment or treatment combinations. A weaker version of intervention would be a form of selection. Rather than interfere too much with the state of nature one may simply select a value of $X$ which is  already in a population, such as selecting a subject of a particular age. Stratification is in this category as is ``matching'', observing (or treating) a collection of subject who are close in terms of some multivariate metric applied to the possible confounders. ``Natural Experiments'' exploit opportunities where Nature has unwittingly designed an experiment for us. For a very thorough compendium of experimental design methodology both as intervention and as selection, see \cite{dean}.

Traditions in agriculture and social sciences have stressed the role of randomisation, and indeed the method has been described as one of the the greatest contributions of statistics to scientific methodology; a major review is \cite{cox} which goes a long way towards updating earlier discussions such as \cite{kemp}. After a long period in which factorial and optimum controlled experiments may be seen to have had a dominant role, influenced by success in product design and quality improvement, randomisation is making a come back, if indeed it ever left the limelight. It is now used extensively outside its traditional areas of clinical trials under the generic term randomised control trials (RCT). Notably, there is a fast growing application to experiments in social media, under the heading AB experiment in on-line marketing, see \cite{koha}, and to socio-technical experiments, such as smart metering in homes and transport, see e.g.~\cite{guz}. Other important developments are in the field of ``big data'', where data is often collected without experimental design being used, so that biases can be a serious impediment to model building, see~\cite{drov,pesce1,pesce2}.

There seems to be no doubt that in nearly all fields the removal of biases in modelling is a major reason to randomise. The question then remains as to whether the randomisation, or rather the randomisation distribution, is to be used in the analysis, e.g., probability statements are made based on the randomisation, or whether randomisation should only be used in the design, e.g. for bias reduction. The latter approach is probably more common and is adopted here. A compromise position is a minimax approach which is closely related to the  use of randomisation in finite population sampling, see~\cite{scott, steng, stig, wynn}.

There is a subtle relationship between randomisation and combinatorial design, which is perhaps closest to the present paper, see~\cite{bailey}. Our purpose is to introduce a specialised but also very general technique,
namely the theory of circuits, already studied in numerical analysis and algebraic statistics. This sets the paper in the sub-areas of ``complex randomisation'', ``structured randomisation'', or ``randomised blocks''.

\section{AB experiments} \label{sect:AB}

Some of the disparate interpretations of randomisation can be understood from a simple AB (RCT) experiment. Using traditional terminology, we want to assess the difference between the effect of two treatments $A$ and $B$ with  effects $\theta_1$ and $\theta_2$, respectively. That is we want to estimate $\phi  = \theta_1 - \theta_2$.

A standard model is to write  for subjects $i$ and $j$ receiving treatments $A$ and $B$, respectively
\begin{eqnarray*}
Y_{1i}  & = & \theta_1 + \delta_{1i},\; i = 1, \ldots, n_1,\\
Y_{2j}  & = & \theta_2 + \delta_{2j},\; j = 1, \ldots, n_2
\end{eqnarray*}
where $n_1,n_2$ are the respective sample sizes and $\delta_{1i}, \delta_{2j}$ are unit effects of other influences, be the errors of measurement or other (hidden) factors.
The naive estimate of the treatment difference is
$$\hat{\phi} =  \hat{\theta}_1 - \hat{\theta}_2 \, . $$
Here the estimates of $\theta_1$ and $\theta_2$ are given by the respective sample means:
$$  \hat{\theta}_1  = \bar{Y}_{1 \cdot} \, ,\;\hat{\theta}_2  = \bar{Y}_{2 \cdot} \, ,$$
where for instance $\bar{Y}_{1\cdot}$ is the usual notation for the average of measurements over group $A$. The standard argument, and this is probably also the common sense argument of non-experts, is that if we randomise  then the  difference between the mean values of the deviations due to other factors will cancel out: $\delta_{1 \cdot} - \delta_{2 \cdot}$, will be approximately zero and will not perturb $\hat{\phi}$. Of course, if $\delta_{1i}, \delta_{2j}$ are random with standard assumptions then $\hat{\phi}$ is both the Least Squares estimate and the best linear unbiased estimate of $\phi$.

A critical question is: what does the model mean, both scientifically and predictively?  What are $\theta_1,\theta_2$ and $\phi$? More precisely, do parameter values refer to the finite population from which the sample was taken or  to which the treatment were applied? Or is there some larger population of which the population of units under study is a subpopulation, such as all present and future subjects who may benefit from a vaccination decision base on the results of the experiment? Or even more metaphysically,  are they part of a larger scientific theory, maybe even a ``crucial experiment'', to decide between two scientific theories?

These questions are important, for example with AB experiments on people using social media. The commercial opportunities in terms of the use of huge (big) data sets come with a risk of bias arising from any number of demographic and operations factors. It is almost impossible to describe the population of social media users but if bias can be removed in some simple way then the estimates can genuinely reflect peoples' choices and behaviour. A naive but rather universal conclusion is something like: after randomisation we can use the model, which is sometimes expressed in a more expert fashion as: make sure you randomise your blocks. On the one hand this paper takes this simple approach, but on the other introduces a special technique, based on circuits, to  decompose an experiment into mutually exclusive  blocks in each of which randomisation can be carried out separately. Some solutions comprise recognisable combinatorial objects, others derive from running the programme {\tt 4ti2}~\cite{4ti2} to obtain the circuits, see subsection \ref{subs:comp}.

After a rather elementary formulation of the problem in the next two sections, we formally define the quest for what we will call {\em valid} randomisation schemes in Section \ref{sect:valid}, followed by a short discussion on analysis in Section \ref{sect:analysis}. Sections \ref{sect:circuits} and \ref{sect:tum} are the main developments, with Section \ref{sect:circuits} describing a sufficient condition that non-negative binary circuit (to be defined) gives a valid randomisation and Section \ref{sect:tum} some special conditions.

\section{Contrasts} \label{sect:contrat}

As we have stated in the AB case, randomisation is particularly suited to situations in which standard estimates are unaffected by a uniform shift of the observations, which is then subtracted out. Consider an experiment giving a random sample $Y_1, \ldots, Y_n$. We have the following:
\begin{defn}
A linear function $Z = \sum_{i=1}^n c_i Y_i$ with fixed coefficients $\{c_i\}$ and data values $\{Y_i\}$ is called an {\em empirical contrast} if $\sum_{i=1}^n c_i = 0$.
\end{defn}

Now consider a standard  regression model in the form
$$Y(x) = \sum_{j=1}^{p} \theta_j f_j(x)  + \epsilon,$$
for functions $\{f_j(x)\}$, $x$ a generic point in some design space $\mathcal X$ and $\epsilon$ a random error with the usual assumptions (zero mean and constant variance).

An experimental design $D = \{x^{(i)}, \; i = 1, \ldots, n\}$, with sample size $|D| =  n$, has design matrix
$$X = \{f_j(x^{(i)}) \},$$
and we express the standard regression set-up, under standard assumptions  by:
$$\mu = \mathbb{E}(Y) = X \theta,$$
where $\mathbb{E}$ is the expectation and $\theta$ the parameter vector.
\begin{defn}
For a standard regression model a \emph{parametric contrast} is defined as the expectation of an empirical contrast.
\end{defn}

To repeat, the basic idea is to divide experiment into disjoint blocks in each of which we randomise, and then combine the results.

\begin{ex}[$2^2$ experiment] We consider a simple example from linear regression, namely a $2^2$ factorial design problem, with $\pm 1$ levels and no replication (for simplicity). We take the model without an interaction
\[
\mathbb{E}(Y)= \theta_0 + \theta_1x_1 + \theta_2x_2 \, ,
\]
so that design matrix is
\[
X = \left[
\begin{array}{rrr}
1 & 1 & 1 \\
 1 & -1 & -1 \\
 1 & 1 & -1 \\
 1 & -1 &  1
\end{array}
\right].
\]
If we randomise a large population and uniformly apply the four combination of the design, $\{ \pm 1, \pm 1 \}$, the potential bias effect will be negligibly small because the estimators of the $\theta$-parameters are unbiased.

But there is an alternative. Split the population into two groups randomise each separately and apply the controls $(x_1,x_2) = \{(1,1),(-1,-1)\}$ to the first group and $\{(1, -1),(-1,1)\}$ to the second group. Then we can estimate $\theta_1 + \theta_2$ from the first group and $\theta_1 - \theta_2$ from the second group. Combining these estimates gives  the same result, except possible the small effect or confounders,  as if we randomised over the whole $2^2$ experiment. Note that the parameters $\theta_1$ and $\theta_2$ and their estimates are already respectively parametric contrasts and empirical contrasts. This can be seen as splitting the $2^3$ experiment into two (randomised) AB experiments.
\end{ex}

\section{Writing a model in contrast form} \label{sect:writing}

In the case of the orthogonal design described above the $X$-matrix takes the form
$$X = [j : X_1],$$
where $j$ is the $n$-vectors of ones, for the constant (intercept) term, and $X_1$ is orthogonal to $j$: $j^T X_1 = 0$. We describe such an $X$ matrix as being {\em in contrast form}. All empirical and parametric contrasts are derived from $X_1$. Thus we can prove the following lemma.
\begin{lem}
For a regression model with $\mu = \mathbb{E}(Y) = \mathbb{E}(\tilde{X}\theta)$, written  in contrast form $\tilde{X} = [j : X_1]$ the set of all parametric contrasts is  $\{c^T \mu: c^T =  c^T X_1 \ \mbox{and} \ j^Tc=0\}$.
\end{lem}
\begin{proof}
This follow  since $\mathbb{E}(c^TY) = c^T[j:X_1]\theta = c^Tj\theta + c^T X_1 \theta = c^T X_1 \theta$.
\end{proof}

Notice that from any model with integer design matrix $X$ it is always possible to derive a reparametrisation with design matrix $\tilde X$ written in contrast form.

\begin{lem}
A design matrix $X$ with column space containing the vector $j = (1,1, \ldots, 1)^T$ can be transformed to contrast form $\tilde{X}=[ j : X_1]$ with the same column space as $X$, where $j^T X_1 =0$.
\end{lem}
\begin{proof}
We can easily determine the reparametrisation which the transformation requires. Starting with:
$$ \tilde{X} \phi = X \theta,$$
we solve for $\phi$:
$$ \phi = (\tilde{X}^T \tilde{X})^{-1} \tilde{X}^T X \theta.$$
\end{proof}

The term contrast is especially prevalent in Analysis of Variance (ANOVA) models, that is models for qualitative factors in which each level of each factor provides a parameter for an additive model. The classical notation, say, for a two-way $I \times J$ table with two factors  is that the additive model would have parameters $\alpha_i, (i=1, \ldots, I)$  and  $\beta_j, (j=1, \ldots, J)$ and the model for the observations $Y_{ij}$ is
$$Y_{ij} = \alpha_i + \beta_j + \epsilon_{ij},$$
where $\{\epsilon_{ij}\}$ are the random errors with standard assumptions. We show this with an example.

\begin{ex}
Let $I=J=2$. By using indicator variables and setting $\theta = (\alpha_1, \alpha_1,\beta_1,\beta_2)^T$  we write the model in regression form, $\mathbb{E}(Y) = X \theta$ where
$$
X = \left[
\begin{array}{rrrr}
1 & 0 & 1 & 0 \\
1 & 0 &  0& 1 \\
0 & 1 & 1 & 0 \\
0 & 1 & 0 & 1 \\
 \end{array}
\right].
$$
This $X$-matrix is {\em not} in contrast form, but it can be transform to one that is:
$$
\tilde{X} = \left[
\begin{array}{rrr}
1 & 1 & 1 \\
1 & 1 &  -1 \\
1 & -1 & 1 \\
1 & - 1 & -1
 \end{array}
\right].
$$
From this, the reparametrisation is:
\begin{eqnarray*}
\phi_0 &  = & \frac{1}{2}( \alpha_1 + \alpha_2 + \beta_1 + \beta_2)\, ,\\
\phi_1 & = & \frac{1}{2}(\alpha_1 - \alpha_2) \, , \\
\phi_2 & = & \frac{1}{2}(\beta_1 - \beta_2) \, .
\end{eqnarray*}
\end{ex}

Note that we have limited the analysis here to the decomposition of $\tilde{X}$ into $[j, X_1]$ since for randomisation we are interested in the decomposition of the vector $j$, but the results in this section and many results about the circuit bases in the next sections could be written in general for a decomposition of $\tilde{X}$ into $[X_2,X_1]$ with $X_2^TX_1=0$.

\section{Valid randomisations} \label{sect:valid}

Using the representation of the design matrix in contrast form we can introduce and analyze the randomisation systems in order to give general answers to the questions stated earlier in Sect.~\ref{sect:AB} in the framework
of AB experiments. The separation into blocks is described by the following definitions.

\begin{defn}
For observations $Y_i,\;(i = 1, \ldots, n)$ a \emph{potential randomisation system} $R$ is a set partition of $\mathcal N = \{1,2, \ldots, n\}$,
namely
a decomposition of $\mathcal N$ into disjoint exhaustive subsets, $R_1, \ldots, R_k$, called blocks, of size 2 or more:
\begin{enumerate}
\item $\cup_{1=1}^k R_i = \mathcal N$
\item $R_i \cap R_j = \emptyset, \, 1 \leq i < j \leq k$
\item $|R_i| \geq 2, i = 1, \ldots, k$
\end{enumerate}
\end{defn}
\begin{defn} For a regression model and experimental design $D_n$ with sample size $n$ and a design matrix in contrast form $[j: X_1]$, a \emph{valid randomisation system} is a potential randomisation system for which all
the associated binary vectors $z^{(i)}= (z_{i,1}, \ldots, z_{i,n})$
$$ z_{i,j} = \left\{
\begin{array}{l}
1,\; i \in R_j \\
0, \; i \in \mathcal N  \setminus  R_j
\end{array}
\right.,
$$
are orthogonal to $X_1$: $ (z^{(i)})^T X_1 = 0$, $i = 1, \ldots, n$.
\end{defn}

The next two examples are familiar in the sense that the orthogonal blocks are easily associated with addition factors or parameters in an orthogonal design. The third example may be less familiar.

\subsection{Factorial fractions}
We consider a $2^3$ factorial experiment for main  effects. The standard $X$-matrix is already in contrast form:
$$X^T = \left[
\begin{array}{rrrrrrrr}
1 & 1 & 1 & 1 & 1 & 1 & 1 & 1 \\
1 & 1 & 1 & 1 & -1 & -1 & -1 & -1 \\
1 & 1 & -1 & -1 & 1 & 1 & -1 & -1\\
1 & -1 & 1 & -1 & 1 &-1 & 1 & -1 \\
 \end{array}
\right]
$$
In addition to a full randomisations there are two different randomisation systems and we list the $R_j$ partitions for each.
\begin{enumerate}
\item $\{1,4,6,7\}, \{2,3,5,8\}\, ;$
\item  $\{1,8\},\{2,7\},\{3,6\}, \{4,5\}\, .$
\end{enumerate}
These two distinct randomisations of this example correspond to familiar decomposition into blocks based on abelian groups (see eg \cite{box}). The first arrives from a  $2^{3-1}$ experiment with defining contrast subgroup in classical notation
$$I = ABC.$$
The second corresponds to the $2^{3-2}$ with subgroup
$$I = AB=BC = AC.$$
For those more familiar with the algebraic design of experiments, these solutions are the point ideal corresponding respectively to the solutions of
$$(1): x_1x_2x_3 = \pm 1,\;\;\mbox{and}\;\;(2): (x_1x_2, x_2 x_3) = (\pm 1, \pm 1).$$

\subsection{Tables and Latin Squares}
Consider an $I \times I$ table with the usual additive model. A Latin square based on the table has the usual definition. If $I=3$ there are two mutually orthogonal Latin squares; in traditional notation:
$$
\begin{array}{ccc}
A & B & C \\
C & A & B \\
B & C & A
 \end{array}\;\;\;\;\;\;
\begin{array}{ccc}
a & b & c \\
b & c  & a \\
 c & a  & b
 \end{array}
$$
Each square gives a different valid randomisation based on the letters. Labelling the observations left-to-right and top-to-bottom the respective blocks are (ignoring commas)
$$\{159, 267, 348\},\;\;\;\;\;\; \{168,249,357\}.$$
We state the general result without proof and in the terminology of this example.
\begin{lem}
For an $I \times I$ additive Analysis of Variance model a set of mutually orthogonal Latin squares provide a set of alternative valid
randomisations.
\end{lem}

\subsection{$k$-out-of-$2k$ choice experiments} \label{k:2k}

Choice experiments are those in which subjects are  asked to score a  selection of attributes from a portfolio of attributes. Models are fitted to experimental data in an effort to discover subjects' (hidden) preference order.

Suppose there are $n=4$ attributes and each subject is offered $k=2$ attributes, labelled $1,2,3,4$. There are six  selection pairs
$$\{1,2\}, \{1,3\}, \{1,4\}, \{2,3\},\{2,4\}, \{3,4\}.$$
An additive preference model has (without replication) the six values $Y_{i,j}$ 
with the model
$$Y_{ij} = \alpha_i + \alpha_j   + \epsilon_{i,j} \quad (i,j = 1,2,3,4 ; i <  j) $$
We are interested in contrast $\alpha_i - \alpha_j$, because their estimates would yield an estimated preference order. In this case:
$$
X = \left[
\begin{array}{rrrr}
1 & 1 & 0 & 0 \\
1 & 0 & 1 & 0 \\
1 & 0 & 0 & 1 \\
0 & 1 & 1 & 0 \\
0 & 1 & 0 & 1 \\
0 & 0 & 1 & 1 \\
 \end{array}
\right].
$$
This gives a choice of $X_1$:
$$
X_1^T = \left[
\begin{array}{rrrrrr}
-1 & 0 & 0 & 0 &0 & 1 \\
0 & -1 & 0 & 0 &1 & 0 \\
0 & 0 & -1 & 1 &0 & 0 \\
\end{array}
\right],
$$
and the randomisation $\{16,25,34\}$.

\section{Analysis} \label{sect:analysis}

The informal approaches we have taken is that, for large samples randomisation has approximately the effect of introducing a block parameter. Our condition of orthogonality in the definition of valid randomisation and as exemplified, has so far ignored the fact that in standard terminology blocks do not have to be orthogonal. Indeed, there is rich theory of balanced incomplete blocks (BIBD) both from combinatorial and from optimal design theory. We note here some basic facts about orthogonal versus non-orthogonal blocks.

\begin{enumerate}
\item For orthogonal designs  we set up a model in which every $j$-vectors is allocated a block parameter, then only under orthogonality is the usual LSE of the $\theta$-parameters the best and there is no bias of these    estimates from the block effects.
\item In the non-orthogonal blocks design  case, if we use the LSE of the $\theta$-parameters assuming that the block parameters are zero, when they are not, then the block parameters introduce bias.
\item In the non-orthogonal blocks case the ``proper'' LSE estimate of the  $\theta$-parameters in the presence of the block parameters, will be unbiased but will have higher variances than in case (2) above (the covariance matrix will Loewner-dominate).

\end{enumerate}

Models with non-orthogonal blocks with a specified block effect, require some effort to model or at least interpret the block affect, for example the effect of day if the experiment is conducted over days.  In such cases a bias model is required. But where bias is caused by hidden, unspecified, confounders, such a bias model seems somewhat artificial. The effects are too artificial to model but  sufficiently present that we prefer orthogonality.

\section{Circuit basis for randomisation} \label{sect:circuits}

In this section, we introduce the circuits of a matrix to analyse the problem of randomisation. We consider the randomisation as the decomposition of the vector $j=(1,\ldots,1)^T$ into binary vectors:
\[
j = j_1 + \ldots + j_k
\]
where each vector $j_h$ is a binary vector satisfying $j_h^TX_1 = 0$, $h=1, \ldots , k$. Such binary vectors $j_h$ are called binary randomisation vectors. Next, we introduce the circuits and their main properties.

Let $A$ be an integer-valued matrix with $d$ rows and $n$ columns. For our purposes, we can assume that $A=X_1^T$. Let $u \in \mathbb{Z}^n$ be an integer-valued vector, $u^+$ be the positive part of $u$, namely $u^+_i = \max(u_i, 0)$, $i = 1, \ldots, n$, and $u^-$ be the negative part of $u$, namely $u^-_i = - \min(u_i, 0)$, $i = 1, \ldots, n$, so that $u = u^+ - u^-$. Moreover, denote with $\mathrm{supp}(u)$ the support of $u$, i.e.,
\[
\mathrm{supp}(u) = \{ i \in \{1, \ldots, n\} \ : \ u_i \ne 0 \} \, .
\]
\begin{defn}
A {\it circuit} of $A$ is an integer-valued vector $u$ in $\ker(A)$ with the following minimality properties:
\begin{enumerate}
\item the binomial $x^{u^+}-x^{u^-}$ is irreducible in the polynomial ring ${\mathbb Q}[x_1, \ldots, x_n]$, where $\mathbb{Q}$ is the set of rational numbers;

\item $u$ has minimal support, i.e., there is no other circuit $v$ with $\mathrm{supp}(v)\subset \mathrm{supp}(u)$.
\end{enumerate}
\end{defn}

\begin{defn}
The set of all circuits of the matrix $A$ is named as the {\it circuit basis} of $A$ and it is denoted with ${\mathcal C}(A)$.
\end{defn}
The circuit basis ${\mathcal C}(A)$ is always finite. The minimal support property gievs rise to a number of interesting properties of ${\mathcal C}(A)$. We recap in the following proposition the special features of the circuits we will use for describing randomisation. For the proofs and further details the reader can refer to~\cite{sturmfels}.

\begin{prop} \label{prop:circ}
Let $A$ be an integer-valued matrix with dimensions $d \times n$ and suppose that $\mathrm{rank}(A)=d$.
\begin{enumerate}
\item The circuit basis ${\mathcal C}(A)$ is subset compatible, i.e., if we consider a matrix $A'$ by selecting $n'<n$ columns, then the     circuit basis of $A'$ is formed by the circuits in ${\mathcal C}(A)$ with support contained in the $n'$ columns.

\item A circuit $u$ in ${\mathcal C}(A)$ has cardinality of the support at most $d+1$.

\item Each vector $v$ of $\ker(A)$ can be written as rational non-negative linear combination of circuits, i.e,
\[
v = \sum_{h=1}^{n-d} q_h u_h \, , \, q_h \in \mathbb{Q}_+
\]
and the $u_h$ are conformal with $v$.
\end{enumerate}
\end{prop}

The term ``conformal'' in the last Item of Prop.~\ref{prop:circ} means that $\mathrm{supp}(u_h^+) \subset \mathrm{supp}(v^+)$ and  $\mathrm{supp}(u_h^-) \subset \mathrm{supp}(v^-)$.

The first key results follow directly form the fact that a circuit lies in $\ker(A)$.
\begin{lem}
Any non-negative binary circuit of $A=X_1^T$ provides a randomisation vector.
\end{lem}

When a non-negative binary circuit $j_1$ gives a valid randomisation, then also $j_2=j-j_1$ is a binary non-negative vector in $\ker(A)$ so that the decomposition $j=j_1+j_2$ is a valid randomization. Note that $j_2$ may be a circuit itself (and in such a case we call $j=j_1+j_2$ a non decomposable randomisation), or not. In the latter case, the vector $j_2$ can be decomposed into the sum of non-negative circuits.

From Proposition \ref{prop:circ}, Item 3, we see that the circuit basis, and in particular the set of non-negative circuits, is the natural tool to find valid non decomposable randomisations. In general, if the vector $j$ can be written as the sum of binary non-negative circuits we have a valid randomisation. The main problem posed in this paper is to provide conditions for when there is a converse, that is to say classes of experimental designs, for which every randomisation vector $j_h$ is a circuit. In the next section we will describe an important class, here we have a useful sufficient condition.

\begin{lem}\label{2c}
If $j_1$ is a non-negative binary randomisation vector with two non-zero elements ($\#\mathrm{supp}(j_1^+)  = 2$), then it is a circuit of $X_1^T$.
\end{lem}

We can see that for every $j_1$-vector in example covered by Lemma~\ref{2c}, there are two rows of $X_1^T$ which have opposite signs. This is the case in Section~\ref{k:2k} which yields the following result.

\begin{lem}
Any $k$-out-of-2$k$ choice experiment is a valid randomisation with blocks of size 2.
\end{lem}

This shows that if we have a valid randomisation comprising binary vectors each with two non-zero binary vectors then it will be found by inspecting the list of all circuits.

\subsection{Computation of circuit}\label{subs:comp}

To find the randomisation systems from the circuit basis, we start from the design matrix $X$, we write it in contrast form $\tilde{X}$, and we extract the contrast matrix $X_1$ as described above. The actual computation of the circuits of the matrix $X_1$ can be done with the software package {\tt 4ti2}, see \cite{4ti2}. In {\tt 4ti2} there is a function called {\tt circuits} which computes the circuits of an integer matrix. The algorithms to compute circuits in {\tt 4ti2} belong to the class of combinatorial algorithms, and thus there is a limitation on the size of the matrices for which the computation of the circuit is actually feasible. In our experiments, problems with a set of points up to 50 are easily processed, but the execution time increases fast with the number of points. However, all the contrast matrices illustrated in this paper have been processed by {\tt 4ti2} in less than 0.1 seconds. {\tt 4ti2} is now available also within the symbolic software {\tt Macaulay2}, see \cite{macaulay2}, and there are {\tt R} packages available which allow the communication between {\tt R} and {\tt Macaulay2}, leading to a flexible use of the symbolic computations into statistical analysis.

\begin{ex}
Using the function {\tt circuits} for the contrast matrix of the $3$-out-of-$6$ problem, we obtain three circuits as expected
\[
\left[
\begin{array}{rrrrrr}
 0 & 0 & 1 & 0 & 0 & 1 \\
 0 & 1 & 0 & 1 & 0 & 0 \\
 1 & 0 & 0 & 0 & 1 & 0 \\
\end{array}
\right].
\]
\end{ex}

\begin{ex}
Computing the circuits for the $2^3$ design with contrasts on the main effects, we obtain the circuits described in the previous sections. The {\tt 4ti2} output consists of $20$ circuits, $6$ of which are non-negative:
\[
\left[
\begin{array}{rrrrrrrr}
0 & 0 & 0 & 1 & 1 & 0 & 0 & 0 \\
0 & 0 & 1 & 0 & 0 & 1 & 0 & 0 \\
0 & 1 & 0 & 0 & 0 & 0 & 1 & 0 \\
0 & 1 & 1 & 0 & 1 & 0 & 0 & 1 \\
1 & 0 & 0 & 0 & 0 & 0 & 0 & 1 \\
1 & 0 & 0 & 1 & 0 & 1 & 1 & 0 \\
\end{array}
\right].
\]
This yields the two randomisation schemes $$\{\{1,4,6,7\},\{2,3,5,8\}\}\;\; \{ \{1,8\},\{2,7\},\{3,6\},\{4,5\}\}$$ already discussed. Here, there is only one valid randomisation based on 2-ers and only one valid randomsation based on 4-ers. (The term $n$-er is a colloquial term for an entity of size $n$.
\end{ex}

With the aid of the circuits we are able to analyse also more complex models where the number of randomisation systems is relatively large.

\begin{ex} \label{ex:5rand}
In the case of $2^4$ design with contrasts on the main effects, the contrast matrix is:
\[X_1^T= \]
\[
\left[\begin{array}{rrrrrrrrrrrrrrrr}
1 & 1 & 1 & 1 & 1 & 1 & 1 & 1 & -1 & -1 & -1 & -1 & -1 & -1 & -1 & -1 \\
1 & 1 & 1 & 1 & -1 & -1 & -1 & -1 & 1 & 1 & 1 & 1 & -1 & -1 & -1 & -1 \\
1 & 1 & -1 & -1 & 1 & 1 & -1 & -1 & 1 & 1 & -1 & -1 & 1 & 1 & -1 & -1 \\
1 & -1 & 1 & -1 & 1 & -1 & 1 & -1 & 1 & -1 & 1 & -1 & 1 & -1 & 1 & -1 \\
\end{array}
\right],
\]
and the situation becomes  more complex. Although $0.02$ seconds are enough to obtain the whole set of $456$ circuits, the non-negative circuits are now $48$ but there are also non-binary circuits with entries equal to $2$.
Selecting the binary circuits reduces to $32$ circuits: $8$ circuits with support on two points give a unique randomisation based on 2-ers, with the remaining $24$ circuits on 4 points we can construct 30 valid randomisations. Each circuit on 4 points is used in 5 possible randomisations. For instance with the circuit
\[
c= \left[\begin{array}{rrrrrrrrrrrrrrrr}
0 &  0  & 0  & 0 & 0 & 1 & 1 & 0 & 1 & 0 & 0 &  1& 0  & 0 & 0  & 0 \\
\end{array}\right]
\]
one can define 5 randomisations, reported in Figure \ref{fig:5rand}.
\end{ex}

\begin{figure}
\begin{center}
\[
\begin{array}{rrrrrrrrrrrrrrrr}
0 & 0 & 0 & 0 & 0 & 1 & 1 & 0 & 1 &  0 &  0 &  1 &  0 &  0 &  0 &  0 \\
0 & 0 & 0 & 0 & 1 & 0 & 0 & 1 & 0 &  1 &  1 &  0 &  0 &  0 &  0 &  0 \\
0 & 1 & 1 & 0 & 0 & 0 & 0 & 0 & 0 &  0 &  0 &  0 &  1 &  0 &  0 &  1 \\
1 & 0 & 0 & 1 & 0 & 0 & 0 & 0 & 0 &  0 &  0 &  0 &  0 &  1 &  1 &  0 \\ \hline
0 & 0 & 0 & 0 & 0 & 1 & 1 & 0 & 1 &  0 &  0 &  1 &  0 &  0 &  0 &  0 \\
0 & 0 & 0 & 1 & 1 & 0 & 0 & 0 & 0 &  0 &  1 &  0 &  0 &  1 &  0 &  0 \\
0 & 1 & 1 & 0 & 0 & 0 & 0 & 0 & 0 &  0 &  0 &  0 &  1 &  0 &  0 &  1 \\
1 & 0 & 0 & 0 & 0 & 0 & 0 & 1 & 0 &  1 &  0 &  0 &  0 &  0 &  1 &  0 \\ \hline
0 & 0 & 0 & 0 & 0 & 1 & 1 & 0 & 1 &  0 &  0 &  1 &  0 &  0 &  0 &  0 \\
0 & 0 & 0 & 1 & 1 & 0 & 0 & 0 & 0 &  1 &  0 &  0 &  0 &  0 &  1 &  0 \\
0 & 1 & 1 & 0 & 0 & 0 & 0 & 0 & 0 &  0 &  0 &  0 &  1 &  0 &  0 &  1 \\
1 & 0 & 0 & 0 & 0 & 0 & 0 & 1 & 0 &  0 &  1 &  0 &  0 &  1 &  0 &  0 \\ \hline
0 & 0 & 0 & 0 & 0 & 1 & 1 & 0 & 1 &  0 &  0 &  1 &  0 &  0 &  0 &  0 \\
0 & 0 & 1 & 0 & 0 & 0 & 0 & 1 & 0 &  1 &  0 &  0 &  1 &  0 &  0 &  0 \\
0 & 1 & 0 & 0 & 1 & 0 & 0 & 0 & 0 &  0 &  1 &  0 &  0 &  0 &  0 &  1 \\
1 & 0 & 0 & 1 & 0 & 0 & 0 & 0 & 0 &  0 &  0 &  0 &  0 &  1 &  1 &  0 \\ \hline
0 & 0 & 0 & 0 & 0 & 1 & 1 & 0 & 1 &  0 &  0 &  1 &  0 &  0 &  0 &  0 \\
0 & 0 & 1 & 0 & 1 & 0 & 0 & 0 & 0 &  1 &  0 &  0 &  0 &  0 &  0 &  1 \\
0 & 1 & 0 & 0 & 0 & 0 & 0 & 1 & 0 &  0 &  1 &  0 &  1 &  0 &  0 &  0 \\
1 & 0 & 0 & 1 & 0 & 0 & 0 & 0 & 0 &  0 &  0 &  0 &  0 &  1 &  1 &  0 \\
\end{array}
\]
\end{center}
\caption{The 5 randomisations for the $2^4$ example containing the circuit $c=(0, 0 ,0 ,0 ,0 ,1 ,1, 0, 1 ,0 ,0 ,1, 0, 0, 0, 0)$.} \label{fig:5rand}
\end{figure}

With a large choice of randomization schemes the problem arises as to which to choose. This is discussed briefly in Section \ref{sect:disc}.

\section{Totally unimodular $X_1$} \label{sect:tum}

Although the factorial design and Latin square examples can be considered well-known, because of orthogonality properties of both, example in Section \ref{k:2k} may be less so. So we may ask what is the property of $X_1^T$ for which the full valid randomisation system can be found as a set of circuits.

\begin{defn}
A \emph{totally unimodular matrix} $A$ is one for which all square submatrices (including itself if square) have determinant $0$, $1$, or $-1$.
\end{defn}

\begin{thm}
Let $A=X_1^T$ be the design/model matrix of  regression model in contrast form and suppose $A$ is totally unimodular. Then every valid randomisation is based on circuits.
\end{thm}
The proof is in two parts. First we need
\begin{lem}\label{TU}
For a totally unimodular matrix $A$ all circuit vectors are binary.
\end{lem}
\begin{proof}
This is based on some known results from the theory of Gr\"obner bases.
Thus,  for $A$ unimodular, the circuits and the Universal Gr\"obner basis are equal. In this statement, the circuits should be seen as represented by the so-called binomials, that is if $u = u^+ - u^-$ are formed by ``dummy'' variables $z_i$ exponents given by $u$:
$$x^{u^+} - x^{u^-}.$$
These binomials generate a toric ideal $I(A)$. This ideal is very widely studied, for example in algebraic statistics it is the starting point for Markov Chain Monte Carlo simulation for testing hypotheses on multinomial contingency tables, see \cite{diaconissturmfels}.

Now, if $A$ is totally unimodular then it is known that the initial ideal $\mathrm{in}(I(A))$ is generated by square-free binomials for any given term-order (required to define a Gr\"obner basis), see \cite{sturmfels}. The initial ideal $\mathrm{in}(I(A))$ of the ideal $I(A)$ is the ideal generated by the leading terms of the polynomials in $I(A)$. Thus, all the binomials in the Universal Gr\"obner basis ${\mathcal U}(I(A))$ have  square-free leading terms.

Finally, the non negative circuits are elements of ${\mathcal U}(I(A))$, viewed as binomials of the form $x^u - 1$. The leading term is always $x^u$, it is square-free and therefore $u$ is binary.
\end{proof}

We now complete the proof with the following.
\begin{lem}
If the contrast matrix $A=X_1^T$ in a regression model is totally unimodular then every non decomposable randomisation vector $j$ is a circuit.
\end{lem}
\begin{proof}
This is by contradiction. Let $j_1$ be a (non-negative binary) non decomposable randomisation vector and suppose it is not a circuit. Since $j_1 \in \ker(A)$, by Prop.~\ref{prop:circ}, Item 3, $j_1$has a representation as a non-negative linear combination of circuits $u_1+ \ldots +u_k$. Take one of such circuits $u_h$. Its support is strictly contained in $\mathrm{supp}(j_1)$ and note that $\#\mathrm{supp}(j_1)-\#\mathrm{supp}(u_h)>1$, because $j_1$ is not a circuit and there are no circuits with support on one point. Moreover, the circuit $u_h$ is binary by Lemma \ref{TU}. So there is a refinement given by $j_1 = u_h + (j_1-u_h)$, which contradicts $j_1$ being non decomposable.
\end{proof}

The most well known example of a totally unimodular matrix is generated by a directed graph $G(V,E)$. The rows are indexed by vertices and the columns by directed edges with the following rule for entries if the edge is $ e= (i \rightarrow j)$ then entries $A_{i,e} = 1,A_{j,e} = -1$ and all other entries in column $e$ are zero. For $A$ to be an $X_1$ matrix we need it to be (row) orthogonal to $j = (1,1, \ldots, 1)$ this requires that for any vertex the number of in-arrows and the number of out-arrows must be the same.

\begin{ex}\label{ex:5points}
 Let $|V| = 5, |E| = 15$ and the directed edges (leaving out commas):
\[
12,13,14,23,24,25,34,35,31,45,41,42,51,52,53 .
\]
In this example $A=X_1^T$ is
\vspace{3mm}
\[  \left[
\begin{array}{rrrrrrrrrrrrrrr}
 1 & 1 & 1& 0& 0 & 0 & 0 & 0 & -1 & 0 & -1 & 0 & -1 & 0 & 0 \\
-1 & 0 & 0& 1& 1 & 1 & 0  & 0 & 0 & 0 & 0 & -1& 0 & -1 & 0 \\
0 & -1 & 0& -1& 0 & 0 & 1 & 1 & 1 & 0 & 0 & 0 & 0 & 0 & -1 \\
0 & 0 & -1& 0& -1 & 0 & -1 & 0 & 0 & 1 & 1 & 1 & 0 & 0 & 0 \\
0 & 0 & 0& 0 &   0 & -1 & 0 & -1 & 0 & -1 & 0 & 0 & 1 & 1 & 1 \\
\end{array}
\right]
\]

\vspace{3mm}

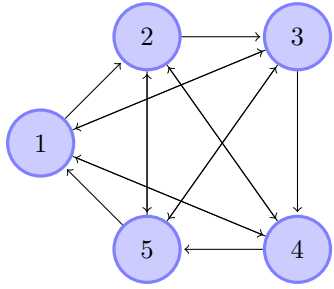
\begin{figure}
\begin{center}
\begin{tikzpicture}[shorten >=1pt,node distance=2cm,auto]
\tikzstyle{every state}=[draw=blue!50,very thick,fill=blue!20]
\node[state] (q_1) {$1$};
\node[state] (q_2) [above right of=q_1] {$2$};
\node[state] (q_5) [below right of=q_1] {$5$};
\node[state] (q_3) [right of=q_2] {$3$};
\node[state] (q_4) [right of=q_5] {$4$};
\path[->] (q_1) edge node {} (q_2);
\path[->] (q_1) edge node {} (q_3);
\path[->] (q_1) edge node {} (q_4);
\path[->] (q_2) edge node {} (q_3);
\path[->] (q_2) edge node {} (q_4);
\path[->] (q_2) edge node {} (q_5);
\path[->] (q_3) edge node {} (q_4);
\path[->] (q_3) edge node {} (q_5);
\path[->] (q_3) edge node {} (q_1);
\path[->] (q_4) edge node {} (q_5);
\path[->] (q_4) edge node {} (q_1);
\path[->] (q_4) edge node {} (q_2);
\path[->] (q_5) edge node {} (q_1);
\path[->] (q_5) edge node {} (q_2);
\path[->] (q_5) edge node {} (q_3);
\end{tikzpicture}
\end{center}
\caption{The directed graph on 5 points in Example \ref{ex:5points}.}  \label{fig:graph2}
\end{figure}

The graph for this example is pictured in Figure \ref{fig:graph2}. For the $X_1$ matrix above, there are 33 nonnegative circuits from a total of 198 circuits: 5 2-ers, 10 3-ers, 10 4-ers, and 8 5-ers.  The valid randomisations we obtained from those circuits are reported in the following table giving the cardinality of the subsets and number $r$ of different choices, classified by the corresponding integer partition.

\vspace{2mm}
\begin{center}
\begin{tabular}{|c|c|}
\hline
randomisation & $r$ \\ \hline
 5+5+5 & 1  \\
5+5+3+2 & 5 \\
5+3+3+2+2 & 5 \\
5+2+2+2+2+2 & 1 \\
4+4+3+2+2 & 10 \\
4+3+2+2+2+2 & 5 \\
3+3+3+2+2+2 & 5 \\ \hline
\end{tabular}
\end{center}
\vspace{2mm}

By the properties of the circuits we know that no proper subset is possible in the previous randomisation, so for instance we know that no randomisation of the form $5+5+3+2$ can share two $5$-ers with the randomisation $5+5+5$. However, the $5+5+5$ shares a $5$-ers with the randomisation $5+2+2+2+2+2$, as shown in Figure \ref{fig:graph}.

\begin{figure}
\begin{center}
\begin{tikzpicture}[scale=0.5]
\filldraw [black] (2,6) circle (4pt)
(4,6) circle (4pt)
(6,6) circle (4pt)
(8,6) circle (4pt)
(10,6) circle (4pt)
(14,5) circle (4pt)
(14,7) circle (4pt)
(16,4) circle (4pt)
(16,8) circle (4pt)
(18,3) circle (4pt)
(18,9) circle (4pt)
(20,2) circle (4pt)
(20,10) circle (4pt)
(22,1) circle (4pt)
(22,11) circle (4pt);
\draw[very thick,blue,rounded corners=3] (1,5.5) rectangle (11,6.5);
\draw[very thick,rotate around={27:(13,6)},blue,rounded corners=3] (13,6) rectangle (24,7);
\draw[very thick,rotate around={-27:(13,6)},blue,rounded corners=3] (13,6) rectangle (24,5);
\draw[red,thick,style=dashed] (6,6) ellipse (5.5 and 1.0);
\draw[red,thick,style=dashed] (14,6) ellipse (0.7 and 2.5);
\draw[red,thick,style=dashed] (16,6) ellipse (0.7 and 3.5);
\draw[red,thick,style=dashed] (18,6) ellipse (0.7 and 4.5);
\draw[red,thick,style=dashed] (20,6) ellipse (0.7 and 5.5);
\draw[red,thick,style=dashed] (22,6) ellipse (0.7 and 6.5);
\end{tikzpicture}
\end{center}
\caption{Two randomisations for the directed graph on 5 points in Figure \ref{fig:graph2}: a $5+5+5$ randomisation and a $5+2+2+2+2+2$ randomisation sharing a 5-er.} \label{fig:graph}
\end{figure}
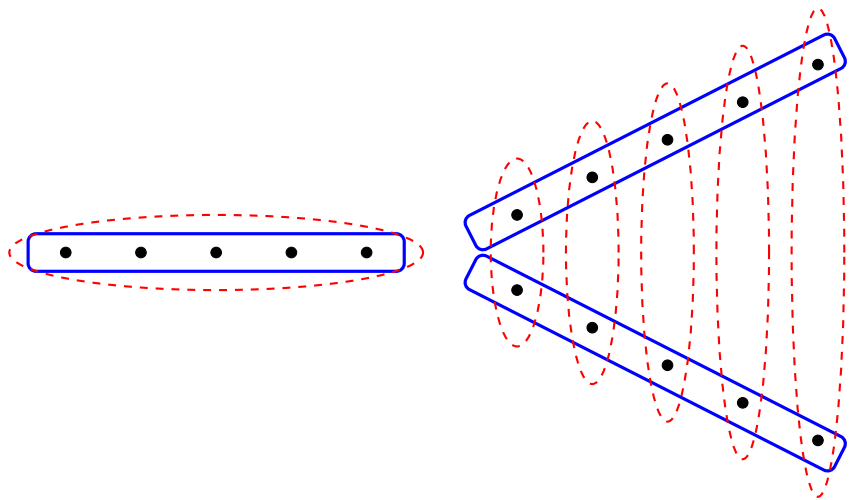

\end{ex}

\section{Discussion} \label{sect:disc}

We can ask a skeptical general question: given the wealth of combinatorial theory to find orthogonal blocks what benefit does the circuit method have? An immediate answer is that it provides, in appropriate cases, the choice of a large, even very large, variety of valid randomisations schemes and under special conditions {\em all} valid randomisations. Weighing designs give some intuition. Historically there are two types. A chemical balance experiment has two pans and compares set of objects. Weighing a set of objects on a one pan weighing machine is very similar to the choice experiments. In the chemical balance the observation itself is an empirical contrast, whereas in the single pan case, we have to reparametrised creating $X_1$ to obtain contrasts, as in the AB experiment. Informally, we could say that in some cases the contrast matrix $X_1$ represents a two-pan experiment embedded in a one pan experiments.

Valid randomisations form a lattice under refinement which we suggest is natural  generalisation of nested randomization. A single non decomposable $j$-vector is a minimal element. A non decomposable valid randomisation corresponds to partition of ${\mathcal N} = \{1,2, \ldots, n\}$. There may be more than one non decomposable valid scheme, as we saw in the $2^3$ example and in the last example. Also relevant is randomisation cost. It may be that a cost function which is related to the structure of the randomization and which is order preserving with respect to the refinement lattice could lead to useful strategies in cases where, as we have seen, the choice of valid randomisations is very large. That is, we have in the background the idea that more refined randomisation is cheaper. There is something of a computational challenge. As we arrive in the ``big data'' era we can expect more sources of bias and mote actual bias. If randomisation is to meet this challenge then we need to extend the theory and the technology of randomisation including fast computation.

There is a considerable literature on sequential randomisation with a model, in the AB case, that subjects (e.g. patients) are awarded treatments A or B on the equivalent of a toss of a fair coin (there is a considerable work on biased coin design which we do not cover). This is an example where the method in this paper should be a cheaper procedure administratively than randomising over a fixed population in order to conduct a more complex randomised block experiment. Note that in the $2^2$ experiment of Example 1 with two blocks of size 2, each block only supplies some of the information. The same for the 4 blocks of size 2 in the $ 2^3$ experiment, whereas for the two $\frac{1}{2}$ fractions of size 4 the parameters can be estimated from each block. In the 2-out-of-4 choice experiments we compare similarly attributes $(1,2)$ v. $(3,4)$, $(1,3)$ v. $(2,4)$ and $(1,4)$ v. $(2,3)$. The two-pan metaphor is useful. The extension to the $k$-out-of-$2k$ example is straightforward and the blocks arise from all ways of splitting $2k$ objects into disjoint set of size $k$. It is likely in our view that sequential and adaptive randomisation will be increasingly important and costs are traded with effectiveness. Their impressive use in CoViD-19 vaccination trial (e.g.~\cite{thor,knoll}) is likely to have a lasting impact.

Finally, some mathematical remarks. The paper could have been written concentrating the link to matroid theory,  because the term {\em circuit} is a term from matroid theory and the circuits presented here form a linear circuit. One matroid  property, for example, is the fact that if the given circuit as defined here it has minimal support in that no vector with whose support is a subset can be a circuit, but should recall that out of the full set of circuits we  select those that have non negative entries. Another mathematical feature which may be useful is that each block of randomisation scheme defined here has an associated permutation group and the full randomisation scheme  generates a subgroup of the full permutation group $S_n$. All possible schemes for a particular example may lead to a complex lattice of subgroups under set partition refinement. The relation between matroids and permutations  group had been studied in \cite{cam}.

\end{document}